\documentclass[12pt,a4paper]{amsart}

\usepackage[centertags]{amsmath}
\usepackage{amsfonts,ifthen,latexsym,amssymb,amsthm, amscd}

\usepackage[T1]{fontenc}
\usepackage[utf8]{inputenc}
\usepackage[american]{babel}

\usepackage[usenames,dvipsnames]{color}

\definecolor{labelkey}{gray}{.20}
\definecolor{refkey}{gray}{.20}
\definecolor{eqkey}{gray}{.20}

\allowdisplaybreaks

\usepackage[american]{babel} 

\usepackage[perpage]{footmisc}

\usepackage{graphicx}

\usepackage{url}
\usepackage{hyperref}
\hypersetup{colorlinks=true,citecolor=blue,filecolor=blue,linkcolor=blue,urlcolor=blue}

\def\mathclap#1{\text{\hbox to 0pt{\hss$\mathsurround=0pt#1$\hss}}}

\newcommand{\beq}{\begin{equation}}
\newcommand{\eeq}{\end{equation}}
\newtheorem{theorem}{Theorem}
\newtheorem*{theorem*}{Theorem}
\newtheorem{lemma}[theorem]{Lemma}
\newtheorem*{lemma*}{Lemma}
\newtheorem{proposition}[theorem]{Proposition}

\newtheorem{conjecture}[theorem]{Conjecture}

\newtheorem{cory}[theorem]{Corollary}
\theoremstyle{definition}

\newtheorem{remark}[theorem]{Remark}
\newtheorem{remarks}[theorem]{Remarks}

\makeatletter
\newtheorem*{rep@theorem}{\rep@title}
\newcommand{\newreptheorem}[2]{%
\newenvironment{rep#1}[1]{%
 \def\rep@title{#2 \ref{##1}}%
 \begin{rep@theorem}}%
 {\end{rep@theorem}}}
\makeatother

\newreptheorem{theorem}{Theorem}
\newreptheorem{lemma}{Lemma}

\newcommand{\al}{{\alpha}}

\newcommand{\be}{{\beta}}

\newcommand{\eps}{{\varepsilon}}
\newcommand{\de}{{\delta}}

\newcommand{\ga}{{\gamma}}
\newcommand{\Ga}{{\Gamma}}

\newcommand{\ka}{{\kappa}}
\newcommand{\la}{{\lambda}}

\newcommand{\si}{{\sigma}}

\renewcommand{\phi}{{\varphi}}

\newcommand\aut{\operatorname{Aut}}

\newcommand\R{\mathbb R}
\newcommand\Q{\mathbb Q}
\newcommand\Z{\mathbb Z}
\newcommand\C{\mathbb C}
\newcommand\N{\mathbb N}

\newcommand\ZZ{\mathbf Z}

\newcommand{\actson}{\curvearrowright}

\newcommand{\wt}[1]{{\widetilde {#1}}}
\newcommand{\wh}[1]{{\widehat {#1}}}
\newcommand{\cal}[1]{{\mathcal #1}}

\makeatletter
\newcommand*\if@single[3]{%
  \setbox0\hbox{${\mathaccent"0362{#1}}^H$}%
  \setbox2\hbox{${\mathaccent"0362{\kern0pt#1}}^H$}%
  \ifdim\ht0=\ht2 #3\else #2\fi
  }
\newcommand*\rel@kern[1]{\kern#1\dimexpr\macc@kerna}
\newcommand*\widebar[1]{\@ifnextchar^{{\wide@bar{#1}{0}}}{\wide@bar{#1}{1}}}
\newcommand*\wide@bar[2]{\if@single{#1}{\wide@bar@{#1}{#2}{1}}{\wide@bar@{#1}{#2}{2}}}
\newcommand*\wide@bar@[3]{%
  \begingroup
  \def\mathaccent##1##2{%
    \if#32 \let\macc@nucleus\first@char \fi
    \setbox\z@\hbox{$\macc@style{\macc@nucleus}_{}$}%
    \setbox\tw@\hbox{$\macc@style{\macc@nucleus}{}_{}$}%
    \dimen@\wd\tw@
    \advance\dimen@-\wd\z@
    \divide\dimen@ 3
    \@tempdima\wd\tw@
    \advance\@tempdima-\scriptspace
    \divide\@tempdima 10
    \advance\dimen@-\@tempdima
    \ifdim\dimen@>\z@ \dimen@0pt\fi
    \rel@kern{0.6}\kern-\dimen@
    \if#31
      \overline{\rel@kern{-0.6}\kern\dimen@\macc@nucleus\rel@kern{0.4}\kern\dimen@}%
      \advance\dimen@0.4\dimexpr\macc@kerna
      \let\final@kern#2%
      \ifdim\dimen@<\z@ \let\final@kern1\fi
      \if\final@kern1 \kern-\dimen@\fi
    \else
      \overline{\rel@kern{-0.6}\kern\dimen@#1}%
    \fi
  }%
  \macc@depth\@ne
  \let\math@bgroup\@empty \let\math@egroup\macc@set@skewchar
  \mathsurround\z@ \frozen@everymath{\mathgroup\macc@group\relax}%
  \macc@set@skewchar\relax
  \let\mathaccentV\macc@nested@a
  \if#31
    \macc@nested@a\relax111{#1}%
  \else
    \def\gobble@till@marker##1\endmarker{}%
    \futurelet\first@char\gobble@till@marker#1\endmarker
    \ifcat\noexpand\first@char A\else
      \def\first@char{}%
    \fi
    \macc@nested@a\relax111{\first@char}%
  \fi
  \endgroup
}
\makeatother

\newcommand{\ov}{\overline}

\newcommand*\mcapinn[2]{\vcenter{\hbox{$\mathsurround=0pt
  \ifx\displaystyle#1\textstyle\else#1\fi\bigcap$}}}

\newcommand*\mcupinn[2]{\vcenter{\hbox{$\mathsurround=0pt
  \ifx\displaystyle#1\textstyle\else#1\fi\bigcup$}}}

\newcommand\tr{\operatorname{tr}}

\newcommand\diag{\operatorname{Diag}}

\renewcommand{\ge}{\geqslant}
\renewcommand{\le}{\leqslant}

\usepackage{graphicx}
\usepackage[centertags]{amsmath}
\usepackage{amsfonts,ifthen,latexsym,amssymb, amsthm, amscd}
\usepackage{enumerate}
\usepackage{tikz}
\usetikzlibrary{arrows}

\newcommand{\MM}{\textsc M}
\renewcommand{\SS}{\textsc S}
\newcommand{\0}{{\texttt{\large 0}}}
\newcommand{\1}{{\texttt{\large 1}}}
\newcommand{\2}{{\texttt{\large 2}}}
\newcommand{\D}{{\texttt{\large D}}}
\renewcommand{\-}{{\texttt{\small $\bullet$}}}
\newcommand{\ppp}{{\texttt{\small $\bullet$}}}
\renewcommand{\u}{\underline}
\newcommand{\s}{\textsc}
\newcommand{\red}{\color{red}}

\begin{document}

\title{Group ring elements with large spectral density}
\author{{\L}ukasz Grabowski}
\let\thefootnote\relax\footnote{\hspace{-18pt}\textit{Email:} \texttt{graboluk@gmail.com}}
\let\thefootnote\relax\footnote{\hspace{-18pt}2010 \textit{Mathematics Subject Classification:} 20C07, 20F65, 57M10.}
\let\thefootnote\relax\footnote{\hspace{-18pt}\textit{Key words and phrases:} $l^2$-invariants, Novikov-Shubin invariants, Determinant approximation, Group rings}
\maketitle
\vspace{-30pt}
\begin{center}
\textit{Mathematics Institute, University of Warwick, Coventry, CV4 7AL, UK}\end{center}
\begin{abstract}
Given $\de>0$ we construct a group $G$ and a group ring element $S\in \Z[G]$ such that the spectral measure $\mu$ of $S$ fulfils $\mu((0,\eps)) > \frac{C}{|\log(\eps)|^{1+\de}}$ for small $\eps$. In particular the Novikov-Shubin invariant of any such $S$ is $0$. The constructed examples show that the best known upper bounds on $\mu((0,\eps))$ are not far from being optimal.


\end{abstract}

\maketitle

\section{Introduction}

The most important technical problem in the general theory of $l^2$-invariants is establishing bounds on the spectral density of group ring elements. Let us illustrate it with three examples. For an introduction to $l^2$-invariants see  \cite{Eckmann_intro} or \cite{Lueck:Big_book} for a more comprehensive treatment. Spectral measures of group ring elements are discussed in the next section.

\vspace{3pt}
\textbf{(i)} The celebrated \textit{L\"uck approximation theorem} states that the $l^2$-Betti numbers of a normal cover of a finite CW-complex are limits of the ordinary Betti numbers of intermediate finite covers, normalized by the cardinality of the fibers. This is an easy corollary of the following statement. For every $C>0$  there exists a function $f\colon \R_+\to \R_+$ such that $f(\eps)\to 0$ when $\eps\to 0$, and such that for every residually finite group $G$, and every self-adjoint $T$ in the integral group ring $\Z[G]$ whose $l^1$-norm is bounded by $C$, we have
\beq\label{eq-approx-bound}
\mu_T((0,\eps)) < f(\eps).
\eeq
In other words, L\"uck approximation follows from \textit{having any uniform bound at all on the spectral density around $0$}. L\"uck \cite{lueck_approximating_l2_invariants_by_their_finite_dimensional_analogues} showed that one can take $f(\eps) := \frac{C'}{|\log(\eps)|}$, where $C'$ depends on $C$.

\vspace{3pt}
\textbf{(ii)} Another $l^2$-invariant, the \textit{$l^2$-torsion}, is not known to be well-defined for arbitrary normal covers. It is well-defined for all normal covers with a given deck transformation group $G$ if and only if for every self-adjoint $T\in \Z[G]$ the integral
$$
\int_{0^+}^1\log(x) \,d\mu_T(x)
$$
is convergent. This is clearly a statement about the density of $\mu_T$ around $0$. The convergence of the integral above was established by Clair \cite{MR1704205} and Schick \cite{MR1828605} for a large class of groups $G$, including all residually-finite ones. As a consequence, using the little-o notation, we have
\beq\label{eq-schick-bound}
\mu_T((0,\eps)) = o\left(\frac{1}{|\log(\eps)|}\right),
\eeq
which is the best \textit{general} upper bound known (see Remark \ref{rem-4}(iii) below).

\vspace{3pt}
\textbf{(iii)} A major open problem, known as the \textit{determinant approximation conjecture}, is the analog of L\"uck approxination for the $l^2$-torsion. It is currently known only when $G$ has the infinite cyclic group $\ZZ$ as a subgroup of finite index (see \cite[Lemma 13.53]{Lueck:Big_book}). It is not difficult to show that  if, under the assumption stated in the example {(i)}, for some $\de>0$ we had 
$$
\mu_T((0,\eps)) < \frac{C'}{|\log^{1+ \de}(\eps)|},
$$
then the determinant approximation conjecture would be true.

\vspace{5pt}

For a long time it has not been known if there actually exist group ring elements whose spectral density is as large as the best known general bound \eqref{eq-schick-bound} suggests. This is reflected in the following conjecture made by Lott and L\"uck.

\begin{conjecture}[Lott-L{\"u}ck \cite{lott-lueck-l2-topological-invariants-of-3-manifolds}]\label{conj-lott-lueck} Let $G$ be a discrete group. For a self-adjoint $T\in \Z[G]$ there exist $C,\eta>0$ such that for sufficiently small $\eps$ we have 
$$
\mu_T((0,\eps)) < C\eps^\eta.
$$
\end{conjecture}

Note that the bound in Conjecture \ref{conj-lott-lueck} is very far away from the best known bound \eqref{eq-schick-bound}: for every $\eta>0$ and sufficiently small $\eps$ we have $\eps^\eta < \frac{1}{|\log(\eps)|}$. However, in this note we show that \eqref{eq-schick-bound} is not too far away from an optimal  bound.

\begin{theorem}\label{thm-main}
For every $\de>0$ there is a group $G_\de$ and a self-adjoint element $S_\de\in\Z[G_\de]$ such that for some constant $C>0$  we have 
\beq\label{eq-claimed-inequality}
    \mu_{S_\de} ((0,\eps_i)) > \frac{C}{|\log(\eps_i)|^{1+\de}}
\eeq
for some sequence of positive $\eps_i$ converging to $0$. In particular, Conjecture \ref{conj-lott-lueck} is false for  $S_\de$.
\end{theorem}

The family $S_\de$ does not invalidate the strategy in the example (iii) above of proving the determinant approximation conjecture, because the supports of $S_\de$ grow when $\de\to 0$. However, in view of how we construct $S_\de$, we state the following conjecture.

We say a function $g\colon \R_+\to \R_+$ is \textit{computable} if there is an algorithm which given $\eps\in \Q$ computes $g(\eps)$.

\begin{conjecture}\label{conj-my}
For every continuous computable function $g\colon \R_+\to \R_+$ such that $g(\eps)\to 0$ when $\eps\to 0$ there exists a group $G$ and $S\in \Z[G]$ such that 
$$
\mu_S((0,\eps_i)) > \frac{g(\eps_i)}{|\log(\eps_i)|}
$$
for some sequence of positive $\eps_i$ converging to $0$.
\end{conjecture}

 Informally, Conjecture \ref{conj-my} claims that the best known general bound \eqref{eq-schick-bound} is optimal. We discuss some arguments in favour at the end of the introduction. Theorem \ref{thm-main} establishes Conjecture \ref{conj-my} for $g(\eps) = \frac{1}{|\log(\eps)|^\de}$, for all $\de>0$. 

\begin{remarks} \label{rem-4}(i) The groups $G_\de$  in Theorem \ref{thm-main} are of the form $[\oplus_\Z \ZZ_2^N \rtimes (\aut(\ZZ_2^N)\wr \ZZ)] \times (\ZZ_2^2\rtimes \aut(\ZZ_2^2))$, where $N$ depends on $\de$. The elements $S_\de$ can also be written down explicitly. 

In the statement of Theorem \ref{thm-main}, the phrase "for some sequence of positive $\eps_i$ converging to $0$" can be replaced by "for all sufficiently small $\eps$". We discuss it in Remark \ref{rem-suff-small}.

 (ii) Another way of phrasing Conjecture \ref{conj-lott-lueck}  is that the so-called \textit{Novikov-Shubin invariants} are always positive. All group ring elements $S_\de$ in Theorem \ref{thm-main} have the Novikov-Shubin invariant equal to $0$. For more information and context we refer to \cite[Chapter 2]{Lueck:Big_book}. Let us note in passing that the only groups $G$ such that the conjecture is known for all $T\in \Q[G]$, are the virtually abelian (see \cite{2013arXiv1310.8564L} and references there) and virtually free groups (see \cite{Sauer(2003)}). 
 
It is interesting to note the contrast between the status of Conjecture \ref{conj-lott-lueck} and the status of the question about the rationality of $l^2$-Betti numbers, i.e.~the \textit{Atiyah conjecture}. Although recently (\cite{arxiv:austin-2009}, \cite{arxiv:grabowski-2010} and \cite{arxiv:pichot_schick_zuk-2010}, \cite{arxiv:lehner_wagner-2010}, \cite{arXiv:grabowski-2010-2}) examples of irrational $l^2$-Betti have been found, it still is very plausible that the Atiyah conjecture holds when the fundamental group is torsion-free. This is not the case with Conjecture \ref{conj-lott-lueck}. Although the groups $G_\de$ in Theorem \ref{thm-main} are not torsion-free, in the appendix we describe a previously unpublished observation of the author and B. Vir\'ag: counterexamples to Conjecture \ref{conj-lott-lueck} with a torsion-free $G$ can be deduced from the mathematical physics literature.

(iii) In \cite{MR2178069} it is shown that the proof of \eqref{eq-schick-bound} from \cite{MR1828605} works for $T\in \Q[G]$, where $G$ is an arbitrary \textit{sofic} group\footnote{We warn the reader that, because of a typo, the statement of \cite[Proposition 6.1(b)]{MR2178069} is false. The correct statement is with $\limsup$ on the left-hand side, and ``$\le$'' instead of ``$=$''.}. We do not discuss sofic groups here, see e.g.~\cite{Pestov_Hyperlinear_and_sofic_groups_a_brief_guide}.

The bound \eqref{eq-schick-bound} is the best known general bound in the following sense. If $G$ is a finitely generated group which is not virtually abelian or virtually free then it is not known if there exists $f$  such that $f(\eps)=o\left(\frac{1}{|\log(\eps)|}\right)$ and such that for every $T\in \Q[G]$ there exists $C_T>0$ such that $\mu_T((0,\eps))<C_T f(\eps)$.

The situation is much worse when $G$ is not assumed to be sofic. Then it is not known if for every $T\in \Q[G]$ there exists a computable function $f_T$ whose values converge to $0$ when $\eps\to 0$ and such that $\mu_T((0,\eps)) \le f_T(\eps)$ for sufficiently small $\eps$. (in particular it is not known if \eqref{eq-schick-bound} holds).

(iv) Even less is known when we consider the real group ring instead of the rational one. If $G$ is not virtually abelian or virtually free, then it is not known if for every $T\in \R[G]$ there exists a computable function $f_T$ whose values converge to $0$ when $\eps\to 0$ and such that $\mu_T((0,\eps)) \le f_T(\eps)$ for sufficiently small $\eps$. The only general result known to the author is from \cite{thom_sofic_groups_and_diophantine_approximation}: the bound $\mu_T((0,\eps)) \le \frac{C}{\sqrt{|\log{\eps}|}}$ holds for $T-\al\in \R[G]$ where $T\in \Q[G]$, $\al\in \R$,  and $G$ is a sofic group. Both here and in the previous remark $\Q$ can be replaced everywhere by the field of algebraic numbers.

(v)  Conjecture \ref{conj-my} does not \textit{contradict} the determinant approximation conjecture, only invalidates the proof strategy mentioned in the example (iii). Furthermore, Li and Thom \cite{MR3110799} establish the determinant approximation in a different setting of compressions  along a F{\o}lner sequence. They do not use any a priori bounds on the spectral density. In fact their approximation result works for an arbitrary element of the group von Neumann algebra, and such elements can have arbitrarily large spectral density around $0$.
\end{remarks} 
\vspace{-5pt}\textbf{Overview of the proof of Theorem \ref{thm-main}.} In the next section we  recall some basics on spectral measures and explain our computational tool. Variants of it were used in the context of computing spectral measures e.g.~in \cite{MR1436310}, \cite{Dicks_Schick}, \cite{Lehner_Neuhauser_Woess:On_the_spectrum_of},  \cite{arxiv:austin-2009}, and \cite{arxiv:pichot_schick_zuk-2010}. For proofs we refer to \cite[Section 2]{arxiv:grabowski-2010}, where a very general version is presented.  For more details on the spectral measures and von Neumann algebras see e.g.~\cite[Chapters 1 and 2]{Lueck:Big_book}, or the extremely readable textbook \cite{Reed_Simon_Methods_of_moder_mathematical_warfare_I}.

In Section \ref{sec-long-chains-suffice} we explain how to deduce Theorem \ref{thm-main} from the existence of a  measure-preserving action $\Ga\actson X$ of a discrete group $\Ga$ on a probability measure space $X$ with certain properties. The most important property is as follows. There are finitely many subsets $X_1,X_2,\ldots, X_n$ of $X$ together with elements $\ga_1,\ldots \ga_n$, such that the following holds. Let $\cal G$ be the graph whose set of vertices is $X$ and with an edge between $x$ and $y$ iff for some $i$ we have $x\in X_i$ and $\ga_i.x=y$. Then there is $d>0$ and a sequence of natural numbers $l_1,l_2,\ldots$, such that the probability that the connected component of $x\in X$ is a line of length at least $l_i$ is more than $\frac{C}{l_i^d}$. The closer $d$ is to $0$, the closer $\de$ in Theorem \ref{thm-main} is to $0$.

In Section \ref{sec-tds} we construct the actions $\Ga\actson X$ with necessary properties.  This is done using the framework of \textit{Turing dynamical systems} introduced in \cite{arxiv:grabowski-2010}. The sets $X_i$ are interpreted as states of a Turing machine and $\ga_i$ as instructions corresponding to a given state. The point is to find a Turing dynamical system whose computations are long, and such that "different computational paths do not interfere with each other", so that the connected components of $\cal G$ above are indeed lines and not trees.

There are no problems in finding Turing dynamical systems with very long computational paths. In fact long enough to prove Conjecture \ref{conj-my}, \textit{if not for the interference condition}. To make sure that different computational paths do not interfere, we use a very specific system which imitates the standard "carry" algorithm of adding numbers. It is very explicit and this allows for checking the interference condition.

However, it seems to the author that the interference condition is more a technical problem than a fundamental obstacle, and this is the reason why we pose Conjecture \ref{conj-my}.

\section{Preliminaries on spectral measures}\label{sec-prelim}

If $D$ is a countable set then $l^2(D)$ is the Hilbert space of all functions $f\colon  D\to \C$ such that $\sum_{d\in D} f(d)\ov{f(d)} <\infty$. The indicator function of $d\in D$ is denoted by $\zeta_d$. If $G$ is a graph whose set of vertices is $V$ then $l^2(G):= l^2(V)$. 

We need to consider Hilbert spaces which are \textit{direct integrals} of other Hilbert spaces - see e.g.~\cite[Subsection 7.4]{Folland:A_course_in_abstract_harmonic_analysis}.

The group ring $\R[G]$ of a group $G$ is the set of formal linear combinations $\sum a_g\cdot g$ where $a_g\in \R$, $g\in G$, and all but finitely many $a_g$ are $0$, together with the obvious addition and multiplication operations. Similarly for $\Z[G]$, $\Q[G]$ and $\C[G]$.

Let $(X,\mu)$ be a probability measure space and let $\Ga\actson X$ be an action of a discrete group by measure-preserving transformations. The result of the action of $\ga\in \Ga$ on $x\in X$ is denoted by $\ga.x$. Consider the Hilbert space $\cal H$ defined as the direct integral
\beq\label{eq-H-decomposes}
\int_X^\oplus l^2(\Ga)\,d\mu(x).
\eeq
Since we will need to specify different fibers in $\cal H$, we denote the copy of $\Ga$ corresponding to $x\in X$ by $\Ga_x$. From now on we will suppress both $X$ and $d\mu(x)$ from all the integrals in this section. As such we could write $\cal H = \int^\oplus l^2(\Ga_x)$.

We have actions of $\Ga$ and of $L^\infty(X)$ on $\cal H$ defined as follows. Both $\Ga$ and $L^\infty(X)$ act in a fiber-preserving fashion. For $\ga\in \Ga$ and $\al\in \Ga_x$ we define $\ga.\zeta_\al := \zeta_{\ga\al}$ and for $f\in L^\infty(X)$ we define $f.\zeta_\al := f(\al.x)\zeta_\al$. The weak completion of the algebra of bounded operators on $\cal H$  generated by $\Ga$ and $L^\infty(X)$ is denoted by $\Ga\ltimes L^\infty (X)$ and is called the \textit{group-measure space construction}. It is an example of a  \textit{von Neumann algebra}. It is equipped with a ${}^\ast$-operation which fulfils
$$
\langle T\zeta_1,\zeta_2\rangle = \ov{\langle \zeta_1,T^\ast \zeta_2\rangle},
$$
for all $\zeta_1,\zeta_2\in \cal H$ and $T\in \Ga\ltimes L^\infty (X)$. If $T=T^\ast$ we say $T$ is \textit{self-adjoint}.

Every $T\in \Ga\ltimes L^\infty (X)$ preserves the fibers of $\cal H$. In other words there is a family $T(x)$, $x\in X$, of operators on $l^2(\Ga_x)$ such that 
\beq\label{eq-T-decomposes}
T = \int_X^\oplus T(x).
\eeq
The \textit{von Neumann trace} of $T$ is defined by $\tau(T):= \int\langle T\zeta_x,\zeta_x\rangle$. If $T$ is a self-adjoint operator on an arbitrary Hilbert space, one associates to it a \textit{projection-valued spectral measure}, i.e.~a function $\pi_T$ from measurable subsets of $\R$ to the projections on the Hilbert space, which has a suitable $\si$-additivity property. If $T\in \Ga\ltimes L^\infty (X)$, the \textit{spectral measure} of $T$ is the measure $\mu_T$ on $\R$ which is equal to the composition  $\tau\circ\pi_T$. It makes sense to compose $\tau$ with $\pi_T$ because for $T\in \Ga\ltimes L^\infty (X)$ and a measurable $D\subset \R$ we  have $\pi_T(D)\in \Ga\ltimes L^\infty (X)$.

A particular case of the group-measure space construction is when $X$ consists of a single point, and $\Ga\actson X$ is the only possible action. The resulting group-measure space construction is the \textit{group von Neumann algebra of $\Ga$} which is denoted by $L(\Ga)$.

We now focus on the situation when $(X,\mu)$ is a compact abelian group with the normalized Haar measure and the action $\Ga\actson X$ is by continuous group automorphisms . Let $A$ be the Pontryagin dual of $X$ (see \cite{Folland:A_course_in_abstract_harmonic_analysis} for more information on the Pontryagin duality). In particular $A$ is a discrete abelian group. We have an embedding 
\beq\label{eq-pontryagin-emb}
\C[A] \to L^\infty(X)
\eeq
 given by the Pontryagin duality. The preimage of $f\in L^\infty(X)$ under this embedding, if it exists, is denoted by $\wh f$. 

We have the \textit{dual action} of $\Ga$ on $A$ which is the unique action which makes the embedding \eqref{eq-pontryagin-emb} equivariant. As such we obtain an embedding
\beq\label{eq-pontryagin-emb-ring}
\C[\Ga\ltimes A] \to \Ga\ltimes L^\infty (X),
\eeq
We have potentially two ways of computing the spectral measure of $T\in \C[\Ga\ltimes A]$. Either as an element of the group von Neumann algebra $L(\Ga\ltimes A)$ or, via the above embedding, as an element of the von Neumann algebra $\Ga\ltimes L^\infty (X)$. 
 
 \begin{lemma}[e.g.~Proposition 2.5 in \cite{arxiv:grabowski-2010}]\label{lem-pass-to-group}
 The spectral measure of $T\in \C[\Ga\ltimes A]$ is the same in $L(\Ga\ltimes A)$ and in $\Ga \ltimes L^\infty (X)$.\qed    
 \end{lemma}

We will produce the groups $G_\de$ and elements $S_\de\in \Z[G_\de]$ in Theorem \ref{thm-main} using the above lemma. The groups will be constructed as $\Ga\ltimes A$ for suitable $A$ and $\Ga$, but the computations showing the desired spectral properties will be carried out in $\Ga\ltimes L^\infty (X)$.

Let us now explain why the spectral computations in $\Ga\ltimes L^\infty(X)$ are sometimes easy to perform. The point is to understand the decomposition \eqref{eq-T-decomposes}. Let $T\in \Ga\ltimes L^\infty (X)$ be given as a finite sum $\sum_{i=1}^n \ga_i\chi_i$, where $\ga_i\in \Ga$ and $\chi_i\in L^\infty (X)$ are indicator functions of some subsets $X_i\subset X$.

Consider the oriented graph $\wt{\cal G}$ associated to $T$ as follows. The set of vertices is $\Ga\times X$ and there is an edge from $(\ga_1,x_1)$ to $(\ga_2,x_2)$  if for some  $i=1,\ldots, n$ we have $\ga_1.x_1\in X_i$ and $\ga_i\ga_1.x_1=\ga_2.x_2$. 

Let $\wt{\cal G}(x)$ be the connected component of $(e,x)$ in  $\wt{\cal G}$. Let $\wt{T}(x)\colon l^2(\wt{\cal G}(x)) \to l^2(\wt{\cal G}(x))$ be the oriented adjacency operator on $\wt{\cal G}(x)$ and for $i=1,\ldots,n$, define 
$$\wt\chi_i(x)\colon  l^2(\wt{\cal G}(x)) \to l^2(\wt{\cal G}(x))
$$
 by $\wt\chi_i(x)\zeta_{(\ga,y)} := \zeta_{(\ga,y)}$ if $\ga.y\in X_i$ and $\wt\chi_i(x)\zeta_{(\ga,y)} :=0$ otherwise. 

Consider now the oriented graph $\cal G_\Ga$ with edge labels from the set $\{\ga_1,\ldots, \ga_n\}$ defined as follows. The set of vertices is $X$, and there is an edge from $x_1$ to $x_2$ with label $\ga$ if for some $i=1,\ldots, n$ we have $\ga=\ga_i$, $x_1\in X_i$ and $\ga_i.x_1=x_2$. Note that there might be multiple edges between any two points of $X$.

Let $\cal G_\Ga(x)$ be the connected component of $x$ in $\cal G_\Ga$. We say $\cal G_\Ga(x)$ is \textit{simply-connected} if multiplying edge-labels along any closed loop gives the trivial element of $\Ga$ (if a loop traverses an edge in the direction opposite to the orientation of the edge, we invert the label). Note that in a simply-connected $\cal G_\Ga(x)$ there are no multiple edges, and self-loops are labelled with the neutral element of $\Ga$.

Finally, let $\cal G(x)$ be the graph which arises from $\cal G_\Ga(x)$ in the following way. If there is an edge from $x_1$ and $x_2$  with label $\ga$ then we replace it by an edge with label equal to $\#\{i\colon x_1\in X_i \text{ and } \ga_i=\ga\}\in \N$.  Let $T(x) \colon l^2(\cal G(x)) \to l^2(\cal G(x))$ be the oriented and labelled adjacency operator on $\cal G(x)$, i.e. the unique operator such that $\langle T(x)\zeta_x,\zeta_y\rangle$ is equal to the sum of labels on edges from $x$ to $y$. Let us  define $\chi_i(x)\colon l^2(\cal G(x)) \to l^2(\cal G(x))$ by $\chi_i(x)\zeta_y := \zeta_y$ if $y\in X_i$ and $\chi_i(x)\zeta_y:=0$ otherwise. 

Let $X^{sf}$ be the subset of $X$ consisting of those $x$ for which  $\cal G_\Ga(x)$ is simply-connected and finite. 

Let $\cal K$ be the Hilbert space
$$
\int^\oplus_{X\setminus X^{sf}} l^2(\wt{\cal G}(x))\oplus \int^\oplus_{X^{sf}} l^2(\cal G(x)).
$$

We have a ${}^\ast$-embedding of  the von Neumann subalgebra $L(T)$ generated by $T, \chi_1,\ldots, \chi_n\in \Ga\ltimes L^\infty(X)$ into $\mathrm B(\cal K)$, induced by 
\begin{gather*}\label{eq-iso}
    T \mapsto \int^\oplus_{X\setminus X^{sf}} \wt T(x) \oplus \int^\oplus_{X^{sf}} T(x)\\ 
     \chi_i \mapsto \int^\oplus_{X\setminus X^{sf}} \wt \chi_i(x) \oplus \int^\oplus_{X^{sf}} \chi_i(x)
\end{gather*}

Under this embedding, $U\in L(T)$ decomposes and we denote the resulting decomposition by
$$
\int^\oplus_{X\setminus X^{sf}} \wt U(x) \oplus \int^\oplus_{X^{sf}} U(x).
$$
A direct computation shows
$$
\tau(U) = \int_{X\setminus X^{sf}} \langle \,\wt U(x)\, \zeta_{(e,x)}, \zeta_{(e,x)}\rangle + \int_{X^{sf}}  \frac{1}{|\cal G(x)|} \tr(U(x)),
$$
where $\tr(U(x)) = \sum_{v\in \cal  G(x)} \langle U(x)\zeta_v,\zeta_v\rangle\,d\mu(x) $ is the standard trace of a finite-dimensional endomorphism.
Accordingly, for self-adjoint $U$ the spectral measure $\mu_U$ decomposes as 
\beq\label{eq-mu3}
\mu_U = \mu^1_U + \mu^2_U.
\eeq 
The measure $\mu^2_U(D)$ of a measurable subset $D\subset \R$ is given by
$$
\mu^2_U(D) = \int_{X^{sf}}   \frac{1}{|\cal G(x)|} \mu_{U(x)}(D) \, d\mu(x).
$$
Note that $\mu_{U(x)}$ is the spectral measure of a \textit{finite-dimensional} endomorphism $U(x)$, i.e.~it is simply the set of eigenvalues of $U(x)$ with multiplicities. This makes 
$\mu^2_U$ relatively easy to compute. The operators $S_\de$ we will construct to prove Theorem \ref{thm-main} will be such that already $\mu^2_{S_\de}$ fulfils the claimed inequality \eqref{eq-claimed-inequality}.

\section{Turing dynamical systems with long computational chains}\label{sec-long-chains-suffice}

It is convenient to construct our group ring elements using  \textit{Turing dynamical systems} introduced in \cite{arxiv:grabowski-2010}, so we start by adapting some definitions from \cite[Section 4]{arxiv:grabowski-2010}. All sets and subsets are measurable whenever it makes sense (and the checks that the sets we define are measurable are straightforward).

Let $(X, \mu)$ be a probability measure space and $\Ga\actson X$ be a measure-preserving action of a countable discrete group $\Ga$.  The result of the action of $\ga\in \Ga$ on $x\in X$ is denoted by $\ga.x$. A {\it dynamical hardware} is the following data: $(X,\mu)$, the action $\Ga\actson X$, and a division $X = \bigsqcup_{i=1}^n X_i$ of $X$ into disjoint measurable subsets. For brevity, we denote such a dynamical hardware by $(X)$.

Suppose we are given a dynamical hardware $(X)$ and we choose three additional distinguished disjoint subsets of $X$, each of which is a union of certain $X_i$: the initial set $I$, the rejecting set $R$, and the accepting set $A$ (all or some of them might be empty). Furthermore, suppose that for every set
$X_i$, we choose one element $\ga_i$ of the group $\Ga$ in such a way that 

(i) the elements corresponding to the sets $X_i$ which are subsets of $R\cup A$ are equal to the neutral element $e$ of $\Ga$, 

(ii) if $\ga_i\neq e$ then $\mu(\{x\in X_i\colon \ga_i.x= x\})=0$. 

A {\it dynamical software} for a given dynamical hardware $(X)$  consists of  the distinguished sets $I, A$ and $R$, and the choice of elements $\ga_i$, subject to the conditions above. 
Each pair $(\ga_i, X_i)$ is referred to as an \textit{instruction}. The set $\ga_i.X_i$ is referred to as the \textit{resulting set} of the instruction $(\ga_i,X_i)$.

Define $T_X: X\to X$ by putting $T_X(x):= \ga_i.x$ for $x\in X_i$. A {\it Turing dynamical system} is a dynamical hardware $(X)$ together with a dynamical software for $(X)$.  We will denote a Turing dynamical system as just defined by $(X, T_X)$.

A \textit{computational chain} is a sequence $x, T_X(x), T^2_X(x),\ldots, T^{l-1}_X(x)$ such that $x\in I$,  $T^{l-1}_X(x)\in A\cup R$, and such that  $l-1$ is the smallest integer for which $T^{l-1}_X(x)\in A\cup R$. Note that it follows that $T^{l}_X(x)=T^{l-1}_X(x)$. The \textit{length} of such a chain is $l$.

Our task is to construct a Turing dynamical system for which there is a subset $Y\subset X$ which is a union of disjoint  computational chains, and such that
\begin{enumerate}[(A)]
\item there exist $d>0$, $C>0$, and infinitely many natural numbers $l$ such that the measure of the set of those $x\in Y$ which are in  computational chains of length at least $l$ is more than $\frac{C}{l^d}$ 
\item $T_X(X\setminus Y)\subset X\setminus Y$.
\item $I\cap T_X(X) =\emptyset$.
\end{enumerate}

We will construct such Turing dynamical systems in Section \ref{sec-tds} (with arbitrarily small $d$). In this section we explain how $(X,T_X)$ gives rise to an operator in the von Neumann algebra $\Ga\ltimes L^\infty(X)$ whose spectral measure is large around $0$.

Let $\chi_i\in L^\infty(X)$ be the indicator function of $X_i$ and let $\chi_I$ be the indicator function of $I$.
\begin{theorem}\label{thm-good-implies-main} Let $(X, T_X)$ be a Turing dynamical system as above which fulfils (A), (B) and (C). Consider the operator 
$$
    T:= \sum_{(\ga_i,X_i)\in P}  \ga_i \chi_i \in \Ga \ltimes L^\infty(X), 
$$
where $P$ is the set of those instructions whose group element is not $e$. Let 
$$
S := 5+ 2(T+T^*) -4\chi_I\in \Ga\ltimes L^\infty(X).
$$
Then there is a constant $C'>0$ and a sequence $\eps_j>0$ converging to $0$ such that 
$$
\mu_S((0,\eps_j)) \ge \frac{C'}{|\log(\eps_j)|^{d+1}}.
$$
\end{theorem}

\begin{proof}
Recall that $Y\subset X$ is a union of disjoint computational chains with properties (A), (B) and (C). The property (B) implies that for $y\in Y$ the connected component $\cal G(y)$ is the directed line:
$$
\bullet \to \bullet \to \ldots \to \bullet.
$$
In particular $Y\subset X^{sf}$. 

Furthermore (C) implies that only the first point of the chain is in the initial set $I$. It follows that $S(y)$ is the adjacency operator on the graph on Figure \ref{fig-1}, where the number of nodes is equal to the length of the computational chain of $y$. 
\begin{figure}[here]
\centering
\begin{tikzpicture}
  [scale=1.1,auto=center]
  \tikzset{edge/.style = {->,> = latex'}}
  \node (n1) at (1,1) {$\bullet$};
  \node (n2) at (3,1)  {$\bullet$};
  \node (n3) at (5,1)  {$\ldots$};
  \node (n4) at (7,1) {$\bullet$};

  \foreach \from/\to in {n1/n2,n2/n3,n3/n4}
    \draw[edge, bend right=23] (\from) to node[below] {$2$} (\to);
  \foreach \from/\to in {n1/n2,n2/n3,n3/n4}
    \draw[edge, bend right=23] (\to) to node[above] {$2$} (\from);
  \foreach \from/\to in {n2/n2,n4/n4}
    \draw[loop above,thick,edge] (\from) to node {$5$} (\to);
    \draw[loop above,thick,edge] (n1) to node {$1$} (n2);
\end{tikzpicture}
\caption{}\label{fig-1}
\end{figure}

Let the adjacency operator on the graph on Figure \ref{fig-1} with $m$ nodes be called $S_m$. The following lemma follows from a linear algebra calculation which we postpone to Section \ref{sec-lemma}.

\begin{lemma}\label{lem-small-eigen} 
There is $0< D<1$  such that for all $m\ge 2$ the adjacency operator $S_m$ has a positive eigenvalue smaller than $D^m$.\qed
\end{lemma}

The equality \eqref{eq-mu3} together with the assumption (A) imply now, for all $l\ge 2$ from the condition (A), 
$$
\mu_S^2((0,D^l)) \ge \frac{1}{l} \cdot \frac{C}{l^d}
$$

Putting $\eps =D^l$ we obtain $\mu_S((0,\eps)) \ge \mu_S^2((0,\eps))\ge \frac{C|\log(D)|^{d+1}}{|\log(\eps)|^{d+1}}$.

\end{proof}

\begin{remark}\label{rem-suff-small}
The proof shows that the sequence of $\eps_j$ in the statement of Theorem \ref{thm-good-implies-main} can be taken to be $\eps_j:=D^{l_j}$, where $D$ is a number in $(0,1)$ and $l_j$ are the lengths fulfilling the assumption (A). In the next section, we show that the sequence $l_j$ can be taken to be $l_j:= L^j$ for a suitable natural number $L$. It follows that for $\eps\in (\eps_{j+1}, \eps_j)$ we have 
\begin{multline*}
\mu((0,\eps))\ge \mu((0,\eps_{j+1})) \ge \frac{C'}{|\log(\eps_{j+1})|^{d+1}} = \frac{C'}{|\log(D^{l_{j+1}})|^{d+1}} =\\
= \frac{C'}{|\log(D^{L\cdot l_{j}})|^{d+1}} = \frac{\frac{C'}{L}}{|\log(\eps_j)|^{d+1}}\ge \frac{\frac{C'}{L}}{|\log(\eps)|^{d+1}}.
\end{multline*}
This shows that the inequality \eqref{eq-claimed-inequality} in Theorem \ref{thm-main} is true, after decreasing the constant, for all sufficiently small $\eps$.
\end{remark}

\section{The "carry" algorithm as a Turing dynamical system}\label{sec-tds}

Informally speaking, the Turing dynamical systems we construct imitate a Turing machine which given the input  $\0\,\0\,\0\,\ldots\,\0$  performs the traditional "carry" algorithm to add $1$ to $\0\,\0\,\0\,\ldots\,\0$ in the $(\D+1)$-ary system for some fixed $\D$, and keeps doing it until it reaches $\D\,\D\,\D\,\ldots\,\D$. The larger $\D$ we choose, the closer to $0$ the constant $d$ in property (A) will be.

There are no problems in obtaining the property (A). Small complications to the idea above arise because of (B) and (C).

Let $N$ be a natural number and let $\MM :=\ZZ_2^N$,  $\SS := \ZZ_2^2$.  They should be thought of as respectively the alphabet and the set of states of a Turing machine. Let us  denote the non-zero elements of $\SS$ by 
\textsc{inc-last-digit}, \textsc{carry}, 
and \textsc{zero-prev-digits}. The element $0\in\SS$ will be called \textsc{void}. 

The element $0\in\MM$ will be denoted by $\-$ (the "empty" symbol), and other elements by $\0, \1,\2,\ldots, \D$ (the order does not matter). The symbol $\D$ is also used to denote the natural number $2^N-2$.

For all pairs of distinct elements $\si,\tau\in \SS\setminus\{\textsc{void}\}$ we choose an automorphism of $\SS$ denoted by $(\si{\to}\tau)\in \aut(\SS)$ which sends $\si$ to $\tau$. Similarly for all $x,y\in \textsc{M}\setminus\{\-\}$ we choose $(x{\to} y)\in \aut(\MM)$ which sends $x$ to $y$.

Let us introduce a notation for subsets  of  $\MM^\ZZ\times \SS$ by giving examples. The set of those $((m_i)_{i\in \Z},s))\in \MM^\ZZ\times \SS$ such that $m_0= \texttt{\large $\textrm x$}$, $m_1= \texttt{\large $\textrm y$}$, $s=\si$ is denoted by 
\begin{equation*}
[\u{\texttt{\large $\textrm x$}}\,\texttt{\large $\textrm y$}, \si].
\end{equation*}
A red symbol means that we allow everything but not that symbol, for example 
\begin{equation*}
 [\u{\texttt{\large $\textrm x$}}\,{\red\texttt{\large $\textrm y$}}, \si]
\end{equation*}
is the subset of $\MM^\ZZ\times \SS$ consisting of those $((m_i),s)$ such that $m_0=\texttt{\large $\textrm x$}, m_1\neq \texttt{\large $\textrm y$}$ and $s=\si$. Finally $\ast$ means that we allow any symbol, for example 
\begin{equation*}
 [\u{\ast}\,{\red\texttt{\large $\textrm x$}}, \si]
\end{equation*}
is the subset of $\MM^\ZZ\times \SS$ consisting of those $((m_i),s)$ such that $m_1\neq \texttt{\large $\textrm x$}$ and $s=\si$.

Concrete elements from the sets as above will be denoted with curly brackets, for example $(\u{\texttt{\large $\textrm x_0$}}\,\texttt{\large $\textrm x_1$}, \si)$ is an element of the set $[\u{\texttt{\large $\textrm x_0$}}\,\texttt{\large $\textrm x_1$}, \si]$.

For the rest of this section $X$ is the compact abelian group $\MM^\ZZ\times \SS$, and $\Ga = (\aut(\MM)\wr\ZZ) \times \aut(\SS)$. The action $\Ga\actson X$ is as follows. $\aut(\SS)$ acts on $\SS$ in the natural way, $\aut(\MM)$ acts on the copy of $\MM$ in the $0$-coordinate of $\MM^\ZZ$ in the natural way, and the generator $t\in\ZZ$ acts by shifting, i.e.~$t.(\u{\texttt{\large $\textrm x$}}\,\texttt{\large $\textrm y$}, \si) = ({\texttt{\large $\textrm x$}}\,\u{\texttt{\large $\textrm y$}}, \si)$.

We define a dynamical hardware $(X)$ by choosing the division $X=\bigsqcup X_i$ to consist of the cylinder sets
\beq\label{eq-cyl-sets}
 [\texttt{\large $\textrm x$} \, \u{\texttt{\large $\textrm y$}}\,\texttt{\large $\textrm z$}, \si],
\eeq
where $\texttt{\large $\textrm x$},\texttt{\large $\textrm y$},\texttt{\large $\textrm z$}\in \MM$, and $\si\in\SS$.

The dynamical software is defined by the following instructions. The sets below are not of the form \eqref{eq-cyl-sets}, but they are finite disjoint unions of sets of the form \eqref{eq-cyl-sets}, so they give rise to instructions in a natural way. The left hand side is the element of $\Ga$ which corresponds to the set on the right hand side.

\small{
\begin{align*}
(\0{\to}\1)&\quad   [\u\0\,\-, \s{inc-last-digit}]\tag{S1}\label{s1}\\
&\ldots\\
(\texttt{\large{D{-}1}}{\to}\texttt{\large{D}}) &\quad [\u{\texttt{\large{D{-}1}}}\,\-, \s{inc-last-digit}]\tag{S2}\label{s2}\\
(\s{inc-last-digit}{\to}{\s{carry}}) &\quad [ \u{\texttt{\large{D}}}\,\-, \s{inc-last-digit}]\tag{S3}\label{s3}\\   
t^{-1} &\quad [\u\D, \s{carry}]\tag{S4}\label{s4}\\
t(\texttt{\large{D{-}1}}{\to}\D)(\s{carry}{\to}\s{zero-prev-digits}) &\quad [\u{\texttt{\large{D{-}1}}}\,{\red\-}, \s{carry}]\tag{S5}\label{s5}\\
t(\texttt{\large{D{-}2}}{\to}\texttt{\large{D{-}1}})(\s{carry}{\to}\s{zero-prev-digits}) &\quad [\u{\texttt{\large{D{-}2}}}\,{\red\-}, \s{carry}]\tag{S6}\label{s6}\\
&\ldots\\
t(\0{\to}\1)(\s{carry}{\to}\s{zero-prev-digits}) &\quad [\u{\0}\,{\red\-}, \s{carry}]\tag{S8}\label{s8}\\
t(\D{\to}\0) &\quad [\u{\D}\,{\red\-}, \s{zero-prev-digits}]\tag{S9}\label{s10}\\
(\D{\to}\0)(\s{zero-prev-digits}{\to}\s{inc-last-digit}) &\quad [\u{\D}\,\-, \s{zero-prev-digits}]\tag{S10}\label{s11}.
\end{align*}}
\normalsize \hspace{-3pt}We will refer to the pairs above also as \textit{instructions}.

If a cylinder set as in \eqref{eq-cyl-sets} is not a subset of one of the sets above, its associated element of $\Ga$ is defined to be the neutral element.  Finally, we define the initial set $I$ to be 
$[\-\,\u{\D}, \s{zero-prev-digits}]$, the accepting set $A$ to be $[\u{\-}, \s{carry}]$, and the rejecting set $R$ to be the empty set.


We will now define a set $Y$ for which the properties (A), (B) and (C) hold. For $j=1,2,\ldots$ let 
\begin{multline*}
Y(j):= [\-\,\u\D\,\D^{j-1}\,\-,\textsc{zero-prev-digits}] \cup \\
\cup T_X([\-\,\u\D\,\D^{j-1}\,\-,\textsc{zero-prev-digits}])\cup \\
\cup T_X^2([\-\,\u\D\,\D^{j-1}\,\-,\textsc{zero-prev-digits}])\cup \ldots,
\end{multline*}
and let $Y:= \bigcup_j Y(j)$.

\begin{proposition}\label{prop-is-good}
The Turing dynamical system $(X,T_X)$ and the set $Y$ just defined fulfil the properties (A), (B) and (C). Furthermore, as $\D$ grows to infinity, we can take $d$ in (A) converging to $0$.
\end{proposition}

\begin{proof}
First we need to check that each $Y$ is indeed a union of disjoint computational paths. We first explicitly check that the map $T_X$ indeed acts as a machine which keeps adding $1$ to its input. We take 
$$x\in [\-\,\u\D\,\D^{j-1}\,\-,\textsc{zero-prev-digits}],$$
denote it by 
$$
(\-\,\u\D\,\D^{j-1}\,\-,\textsc{zero-prev-digits})
$$
and write its trajectory under $T_X$ in Figure 2. Now a direct examination shows that each trajectory ends in the accepting set  $[\u{\-}, \s{carry}]$ (so that each trajectory is a computational chain) and that the trajectories are  indeed pair-wise disjoint.

\begin{figure}[here]
\centering
\begin{tabular}{|l|l|}
$(\ppp\,\u\D\,\D^{j-1}\,\ppp,\textsc{zero-prev-digits})$          & $(\ppp\,\0^{j-2}\,\1\,\u\0\,\ppp, \textsc{inc-last-digit})$    \\
$(\ppp\,\0\,\u\D\,\D^{j-2}\,\ppp,\textsc{zero-prev-digits})$      &  $\ldots$     \\
$\ldots$                                                    & $(\ppp\,\D^{j-1}\,\u\0\,\ppp, \textsc{inc-last-digit})$   \\
$(\ppp\,\0^{j-1}\,\u\D\,\ppp, \textsc{zero-prev-digits})$          & $(\ppp\,\D^{j-1}\,\u\1\,\ppp, \textsc{inc-last-digit})$            \\
$(\ppp\,\0^{j-1}\,\u\0\,\ppp, \textsc{inc-last-digit})$             & $\ldots$           \\
$(\ppp\,\0^{j-1}\,\u\1\,\ppp, \textsc{inc-last-digit})$             & $(\ppp\,\D^{j-1}\,\u\D\,\ppp, \textsc{inc-last-digit})$        \\
$\ldots$                                                    &  $(\ppp\,\D^{j-1}\,\u\D\,\ppp, \textsc{carry})$    \\
$(\ppp\,\0^{j-1}\,\u\D\,\ppp, \textsc{inc-last-digit})$            & $(\ppp\,\D^{j-2}\,\u\D\,\D\ppp, \textsc{carry})$           \\
$(\ppp\,\0^{j-1}\,\u\D\,\ppp, \textsc{carry})$                     & $\ldots$        \\
$(\ppp\,\0^{j-2}\,\u\0\,\D\,\ppp, \textsc{carry})$                 & $(\ppp\,\u\D\,\D^{j-1}\,\ppp,\textsc{carry})$        \\
$(\ppp\,\0^{j-2}\,\1\,\u\D\,\ppp, \textsc{zero-prev-digits})$       & $(\u\ppp\,\D\,\D^{j-1}\,\ppp,\textsc{carry})$               \\
\end{tabular}
\caption{}\label{fig-2}
\end{figure}


It is straightforward to show that  the condition (A) is fulfilled: since any number between $\0\,\0\,\ldots\,\0$ ($j$ times) and $\D\,\D\,\ldots\,\D$ ($j$ times) appears on the tape, it follows that the computational chain as in Figure \ref{fig-2} has length at least $(\D+1)^j$. For the same reason, and since the measure on $X$ is the product measure, we have that the measure of $Y(j)$ is at least$\frac{1}{|\SS|} \left(\frac{\D+1}{\D+2}\right)^j\left(\frac{1}{\D+2}\right)^2$. 

It follows that in the condition (A) we can take $C:= (|\SS|(\D+2))^{-2}$, and any $d$ such that $\frac{1}{(\D+1)^d} < \left(\frac{\D+1}{\D+2}\right)$.


Conditions (B) and (C) follow from a careful case-by-case analysis of our dynamical software. Let us start with (C). We have to check that the resulting set of every instruction  \eqref{s1}-\eqref{s11} intersects trivially the initial set $I=[\-\,\u{\D}, \s{zero-prev-digits}]$. Just by considering the resulting state (i.e.~the $\SS$-coordinate of the resulting set),  we obtain the trivial intersection of $I$ with the resulting sets of instructions \eqref{s1}-\eqref{s4} and \eqref{s11}.

The resulting set of \eqref{s5} is
$$
t(\texttt{\large{D{-}1}}{\to}\D)(\s{carry}{\to}\s{zero-prev-digits}).[\u{\texttt{\large{D{-}1}}}\,{\red\-}, \s{carry}],
$$
which is equal to $[\D\,{\u{\red\-}}, \s{zero-prev-digits}]$, which clearly has trivial intersection with  $[\-\,\u{\D}, \s{zero-prev-digits}]$. Instructions \eqref{s6}-\eqref{s8} are checked similarly.

For \eqref{s10} the resulting set is 
$$
t(\D{\to}\0) . [\u{\D}\,{\red\-}, \s{zero-prev-digits}],
$$
which is equal to   $[\0\,\u{\red\-}, \s{zero-prev-digits}]$, which also intersects the initial set trivially.


Let us now check the  condition (B). We just saw that the first element in Figure \ref{fig-2} has no preimage under $T_X$. On the other hand it is clear that every other element in Figure \ref{fig-2} has exactly one preimage under $T_X$ which is also an element of $Y$ (namely the preceding element in Figure \ref{fig-2}). It follows that to check the condition (B) it is enough to show that for every pair $((\be_1,B_2), (\be_2, B_2))$ of different instructions from among $\eqref{s1}-\eqref{s11}$ we have 
$$
\be_1.B_1 \cap \be_2.B_2 = \emptyset.
$$

We only consider those pairs of instructions where it is not enough to look at the resulting state to show the trivial intersection. For example we would not consider the pair \eqref{s1} and \eqref{s3}, because 
$$
(\0{\to}\1).[\u\0\,\-, \s{inc-last-digit}] \subset [\s{inc-last-digit}]
$$
and 
$$
(\s{inc-last-digit}{\to}{\s{carry}}). [ \u{\texttt{\large{D}}}\,\-, \s{inc-last-digit}] \subset [\s{carry}],
$$
and so the intersection is trivial.

As such we consider three cases, depending on what is the resulting state.

\textbf{Case 1.} The resulting state is \s{inc-last-digit}.

We have to consider the instructions \eqref{s1}-\eqref{s2} and \eqref{s11}. If both instructions are from among \eqref{s1}-\eqref{s2}, it is enough to look at the underlined symbol. Thus we consider a pair consisting of  \eqref{s11} and one of the instruction \eqref{s1}-\eqref{s2}. As for \eqref{s11}, we have
\begin{multline*}
(\D{\to}\0)(\s{zero-prev-digits}{\to}\s{inc-last-digit}) . [\u{\D}\,\-, \s{zero-prev-digits}]= \\
= [\u{\0}\,\-, \s{inc-last-digit}],
\end{multline*}
which is disjoint from the resulting sets of the instructions \eqref{s1} and \eqref{s2}, as they are both contained in $[{\red\u\0},\s{inc-last-digit}]$.

\textbf{Case 2.} The resulting state is \s{carry}.

We have to consider the instructions \eqref{s3} and \eqref{s4}. For \eqref{s3} we have
$$
(\s{inc-last-digit}{\to}{\s{carry}}). [ \u{\texttt{\large{D}}}\,\-, \s{inc-last-digit}] = [ \u{\texttt{\large{D}}}\,\-, \s{carry}], 
$$
and for \eqref{s4} we have 
$$
t^{-1} . [\u\D, \s{carry}] = [\u{\ast}\, \D,\s{carry}],
$$
and so the intersection is trivial because $\D \neq \ppp$.

\textbf{Case 3.} The resulting state is \s{zero-prev-digits}.

We have to consider the instructions \eqref{s5}-\eqref{s8} and \eqref{s10}. If both instructions are from among \eqref{s5}-\eqref{s8}, considering the symbol immediately to the left of the underlined symbol establishes the trivial intersection. 

For the case when one of the instruction is \eqref{s10}, note that the resulting sets of \eqref{s5}-\eqref{s8} are all contained in 
$$
    [{\red\0}\,\u\ast,\s{zero-prev-digits}],
$$
and the resulting set of \eqref{s10} is contained in $[0\,\u\ast,\s{zero-prev-digits}]$.

\end{proof}

We have all the ingredients to finish the proof of Theorem \ref{thm-main}.

\textit{Proof of Theorem \ref{thm-main}.}
For every $d>0$, Proposition \ref{prop-is-good} gives a Turing dynamical system $(X,T_X)$ which fulfils the properties (A), (B) and (C).  Theorem \ref{thm-good-implies-main} gives us an  operator $S\in \Ga\ltimes L^\infty(X)$ such that for some constant $C'>0$ and a sequence $\eps_j>0$ converging to $0$ we have
$$
\mu_S((0,\eps_j)) \ge \frac{C'}{|\log(\eps_j)|^{d+1}}.
$$.
 
Furthermore, we constructed our Turing dynamical system in such a way that $(X,\mu)$ is a compact abelian group with the normalized Haar measure, the action $\Ga\actson X$ is by continuous group automorphisms, and the indicator functions of $X_i$ are in the image of the Pontryagin duality  embedding $\Q[A]\to L^\infty (X)$, where $A$ is the Pontryagin dual of $X$. As such, Lemma \ref{lem-pass-to-group} gives us an element $\wh S\in \Q[\Ga\ltimes A]$ whose spectral measure is the same as the spectral measure of $S$. This finishes the proof. 
\qed

\section{Proof of Lemma \ref{lem-small-eigen}}\label{sec-lemma}

We finish this note by proving Lemma \ref{lem-small-eigen}. Let us recall the statement for reader's convenience.

\begin{figure}[here]
\centering
\begin{tikzpicture}
  [scale=1.1,auto=center]
  \tikzset{edge/.style = {->,> = latex'}}
  \node (n1) at (1,1) {$\bullet$};
  \node (n2) at (3,1)  {$\bullet$};
  \node (n3) at (5,1)  {$\ldots$};
  \node (n4) at (7,1) {$\bullet$};

  \foreach \from/\to in {n1/n2,n2/n3,n3/n4}
    \draw[edge, bend right=23] (\from) to node[below] {$2$} (\to);
  \foreach \from/\to in {n1/n2,n2/n3,n3/n4}
    \draw[edge, bend right=23] (\to) to node[above] {$2$} (\from);
  \foreach \from/\to in {n2/n2,n4/n4}
    \draw[loop above,thick,edge] (\from) to node {$5$} (\to);
    \draw[loop above,thick,edge] (n1) to node {$1$} (n2);
\end{tikzpicture}
\caption{}\label{fig-app}
\end{figure}

\begin{replemma}{lem-small-eigen} 
There is $0<D<1$ such that for all $m\ge 2$ the adjacency operator $S_m$ has a positive eigenvalue smaller than $D^m$.\qed
\end{replemma}

\begin{proof}
The matrix of the adjacency operator of the graph on Figure \ref{fig-app} with $m$ nodes  is 
$$
U_m :=  \left( \begin{array}{cccccc}
1 & 2 &  &  &   \\
2 & 5 & 2 &   &   \\
 & 2 & 5  &  &   \\
& & & \ldots & & \\
 & & &  & 5 &2 \\ 
  & & &  &2 & 5
\end{array} \right)
$$

We check by hand that $\det(U_1) = \det(U_2) =1$, and we prove inductively that $\det(U_m)=1$ for all $m>1$ by expanding the determinant along the last row.

Note that the matrix 
$$
V_m :=  \left( \begin{array}{cccccc}
0 & 2 &  &  &   \\
2 & 0 & 2 &   &   \\
 & 2 & 0  &  &   \\
& & & \ldots & & \\
 & & &  & 0 &2 \\ 
  & & &  &2 & 0
\end{array} \right)
$$
is equal to $2V'_m$, where $V'_m$ is the adjacency matrix of a bipartite graph of maximal degree $2$. Recall that the spectrum of a bipartite graph is symmetric (\cite[Proposition 8.2]{MR1271140}). It follows that the eigenvalues of $V_m$ are contained in $[-4,4]$ and there are as many eigenvalues in $[-4,0]$ as in $[0,4]$, and so the eigenvalues of $U_M + \diag(4,0,\ldots,0)$ are contained in $[1,9]$, and at least half of them are contained in $[5,9]$ ("at least" because of the eigenvalue "5").

Let us recall the Weyl's inequality for rank one perturbations (e.g.~\cite[Theorem 4.3.4]{MR1084815}). Let $\la_1\le \ldots \le \la_m$ be the eigenvalues of $U_m$ and $\ka_1\le \ldots, \ka_m$ be the eigenvalues of $U_m+\diag(4,0,\ldots, 0)$. Then for $k=1,\ldots, m-1$ we have
$$
    \ka_k \le \la_{k+1}.
$$

It follows that $1\le \la_2$, and $5 \le \la_{\lfloor m/2 \rfloor+1}$. Estimating ${\lfloor m/2 \rfloor}$ by $\frac{m}{3}$ from below we obtain
$$
1=\det(U_m) = \la_1\la_2\ldots \la_m > \la_1 \cdot 5^\frac{m}{3},
$$
and therefore we can take $D:= 5^{\frac{1}{3}}$.
\end{proof}

\begin{remark}
We note in passing that the matrices $U_m$ from the proof of Lemma \ref{lem-small-eigen} give a counterexample to the determinant approximation conjecture in the context of graph sequences convergent in the \textit{Benjamini-Schramm} sense (see \cite[Chapter 19]{MR3012035} for the definition). Such counterexamples have been known among experts for a fairly long time, although none seem to have been published. This particular counterexample is due to  L. Lov{\'a}sz.

Indeed, on the one hand we have that $\det(U_m)= 1$ for all $m$. On the other hand, we have established in the proof of Lemma \ref{lem-small-eigen} that at least third of the eigenvalues of $\det(U_m + \diag(4,0,\ldots,0))$ are larger or equal to $5$, and all of them are larger  or equal to $1$. This shows $(\det(U_m + \diag(4,0,\ldots,0)))^{\frac{1}{m}} \ge 5^{\frac{1}{3}}$. We conclude that given a sequence of finite graphs which is convergent in the Benjamini-Schramm sense to a Cayley graph of $\ZZ$, it is not true that the normalized determinants of the adjacency matrices on the finite graphs  converge to the Fuglede-Kadison determinant of the adjacency operator the Cayley graph.

\end{remark}
\section*{Appendix - A counterexample to Conjecture \ref{conj-lott-lueck} from mathematical physics}
In this appendix we present the following unpublished observation of the author and B. Vir\'ag: a counterexample to Conjecture \ref{conj-lott-lueck} can be also deduced from the mathematical physics literature. 

Let us specialize to the following situation. Let $\ZZ\actson X$ be an essentially free action (i.e.~free on a subset of full measure). Let $t\in \ZZ$ be a fixed generator, let $F\colon X \to \R$ be a bounded measurable function and let $T\in \Ga\ltimes L^\infty (X)$ be 
\beq\label{eq-nonrandom-operator}
T:= tF +Ft^{-1} = tF + t^{-1}(t.F).
\eeq

For $x\in X$ let $T(x)\colon l^2(\Z)\to l^2(\Z)$ be defined on the standard basis vectors as  
\beq\label{eq-random-operator}
T(x)\zeta_k := F(t^k.x)\zeta_{k+1} + F(t^{k-1}.x)\zeta_{k-1}.
\eeq 

Let us identify $l^2(\ZZ)$ with $l^2(\Z)$ in the natural way. As explained in Section \ref{sec-prelim}, $T$ is an operator on $\int^\oplus_X l^2(\Z) \, d\mu(x)$ which preserves the fibers $l^2(\Z)$. A direct check shows that $T$ decomposes as 
$$
\int^\oplus_X T(x)\,d\mu(x).
$$

Since $\tau(T) = \int_X  \langle T(x)\zeta_0,\zeta_0\rangle$, we have the following formula for the spectral measure of $\mu_T$ of $T$:
\beq\label{eq-expect-rooted}
\mu_T (D) = \int_X \langle P(x,D)\zeta_0, \zeta_0\rangle\,d\mu(x),
\eeq
where $D\subset \R$ is a measurable set and $P(x,D)$ is the spectral projection of $T(x)$ corresponding to the set $D$ (see e.g. \cite[Lemma 1.9]{arxiv:grabowski-2010}).  

The measure $D\mapsto \langle P(x,D)\zeta_0, \zeta_0\rangle$ on $\R$ is called the \textit{rooted spectral measure} of $T(x)$.  Equation \eqref{eq-expect-rooted} can be concisely summarized as follows. The family $T(x)$, $x\in X$, is a \textit{random operator}, and the spectral measure of $T$ is equal to the \textit{expected rooted spectral measure of $T(x)$}.

We specialize further as follows. Let $(S,\nu)$ be a compact abelian group with the normalized Haar measure, let $A$ be its Pontryagin dual, and let $X=\prod_\Z S$ together with the product measure. Consider the action  $\ZZ \actson X$ given by shifting the coordinates. Let $f\colon S \to \R$ be a function in the image of the Pontryagin duality map $\Q[A]\to L^\infty(S)$, and let $F\colon X \to \R$ be defined as $F((x_i)) := f(x_0)$. 

For $x=(x_i)\in X$ the formula \eqref{eq-random-operator} becomes
\beq\label{eq-random-operator-2}
T(x)\zeta_k = f(x_k)\zeta_{k+1} + f(x_{k-1})\zeta_{k-1}.
\eeq 
Such families of operators have been studied in mathematical physics at least since \cite{MR0059210}.  We provide a small dictionary. First, usually there would be no reference to a concrete measure space $X$. The equation \eqref{eq-random-operator-2} would be written as
\beq\label{eq-random-ham}
H\zeta_k = W(k)\zeta_{k+1} + W(k-1)\zeta_{k-1},
\eeq
together with the assumption that the numbers $W(k)$ are random, independent, and distributed according to some probability measure $\phi$ on $\R$. In the case of \eqref{eq-random-operator-2} the measure $\phi$ is the push-forward of $\nu$ through $f$, i.e.~$\phi(D):=\nu(f^{-1}(D))$.

The operator $H$ is the \textit{Hamiltonian} associated to the \textit{one-dimensional disordered chain} whose disorder is \textit{i.i.d.}, distributed according to the \textit{law} $\phi$.  The spectral measure of the original $T\in \ZZ\ltimes L^\infty(X)$ given by \eqref{eq-nonrandom-operator} is referred to as the \textit{expected spectral measure}, or the \textit{expected density of states}, in order to differentiate it from the rooted spectral measures of the operators $T(x)$, $x\in X$.

In his very impressive article,  Dyson \cite{MR0059210}  considered a disorder distributed according to some specific measure $\phi$. He showed that for that specific disorder the expected spectral measure $\mu_H$ has the property
\beq\label{eq-dyson-sing}
\mu_H((0,\eps)) \approx \frac{1}{|\log(\eps)|^2},
\eeq
where $\approx$ means that the ratio of both sides approaches a positive constant when $\eps\to 0$. Since then, such behaviour is referred to as the \textit{Dyson's singularity}. It is the most eminent qualitative difference between the disordered chains and the situation without a disorder, i.e.~when all $W(k)$ are equal to $1$.

We cannot use Dyson's result to give a counterexample to Conjecture \ref{conj-lott-lueck}, because the measure $\phi$ he considered is not supported in a bounded interval. Accordingly $\phi$ is not a push-forward of $\nu$ through  $F\in L^\infty(X)$ (since, by definition, such $F$ must be \textit{bounded}, and so the push-forward of \textit{any} measure through $F$ is supported in a bounded interval). In particular, the Hamiltonian considered by Dyson is not an element of the von Neumann algebra $\ZZ\ltimes L^\infty (X)$.

However, Dyson's singularity \eqref{eq-dyson-sing} has been conjectured to appear for arbitrary measures $\phi$ which have a non-atomic part. So far, the exact form \eqref{eq-dyson-sing} has been confirmed only via heuristic arguments (e.g.~\cite{PhysRevB.18.569}). However, the following has been rigorously established. Recall that $G\in L^1(\R)$ is  the  \textit{density} of a measure $\phi$ if for every measurable $D\subset \R$ we have $\phi(D) = \int_D G(x) \,dx$.

\begin{theorem}[Campanino and Perez \cite{MR1014114}]
Let  $W$ in \eqref{eq-random-ham} be distributed according to a measure $\phi$ with a continuous density supported in an interval $(a, b)$, where $b>a>0$. Then the expected spectral measure $\mu_H$ of \eqref{eq-random-ham} has the property that 
$$
\mu_H((0,\eps)) \ge \frac{C}{{|\log(\eps)|^3}}
$$
for some $C>0$ and all sufficiently small $\eps$.\qed
\end{theorem}

As such, to provide a counterexample to Conjecture \ref{conj-lott-lueck}, it is enough to find a compact abelian group $(S,\nu)$ together with its Pontryagin dual $A$ and $\wh f\in \Q[A]$ such that the push-forward of $\nu$ through $f\colon S\to \R$ has a continuous density supported in some interval $(a, b)$, where $b>a>0$. A simple exercise in the Fourier transform shows that we can take $A=\ZZ^3$ and $\wh f = s_1+s_1^{-1}+s_2+s_2^{-1}+s_3+s_3^{-1}+7$.

Together with Lemma \ref{lem-pass-to-group} we obtain the following.

\begin{cory}
Let $T\in \Q[\ZZ^3\wr \ZZ]$ be given as 
$$
T:= (s_1+s_1^{-1}+s_2+s_2^{-1}+s_3+s_3^{-1}+7)t + t^{-1}(s_1+s_1^{-1}+s_2+s_2^{-1}+s_3+s_3^{-1}+7).
$$
Then the spectral measure $\mu_T$ of $T$ has the property that 
$$
\mu_T((0,\eps)) \ge \frac{C}{{|\log(\eps)|^3}}
$$
for some $C>0$ and all sufficiently small $\eps$.\qed
\end{cory}

\begin{remarks}
At least one other example of a group ring element with an interesting spectral measure can be exhibited using mathematical physics literature. Let $a$ be the non-trivial element of $\ZZ_2$ and let $T\in \R[\ZZ_2\wr \ZZ]$ be given as $T:= t+t^{-1} + \be\cdot a \in \R[\ZZ_2\wr \ZZ]$, where $\be \in \R$. Then \cite{MR914426} implies, repeating the discussion above, that when $\be$ is large enough, the spectral measure of $T$ has only singularly continuous part, i.e.~it has no atoms and it is not possible to write it as $f\la$, where $f$ is a measurable function and $\la$ is the Lebesgue measure.

This is the only example of such  behaviour known to the author. It has been conjectured that no singularly continuous part appears in the spectral measure of $T$ when $T$ is an element of the group ring of a torsion-free group (\cite{thom_sofic_groups_and_diophantine_approximation}). Unfortunately the mathematical physics literature does not seem to provide a counterexample to that conjecture.
\end{remarks}

\vspace{5pt}
\textbf{Acknowledgements.}
The author would like to thank Holger Kammeyer, Thomas Schick and an anonymous referee for valuable comments.




\bibliographystyle{alpha}
\bibliography{bibliografia}

\begin{thebibliography}{{Gra}10}

\bibitem[Aus13]{arxiv:austin-2009}
Tim Austin.
\newblock Rational group ring elements with kernels having irrational
  dimension.
\newblock {\em Proc. Lond. Math. Soc. (3)}, 107(6):1424--1448, 2013.

\bibitem[Big93]{MR1271140}
Norman Biggs.
\newblock {\em Algebraic graph theory}.
\newblock Cambridge Mathematical Library. Cambridge University Press,
  Cambridge, second edition, 1993.

\bibitem[BVZ97]{MR1436310}
C{\'e}dric B{\'e}guin, Alain Valette, and Andrzej Zuk.
\newblock On the spectrum of a random walk on the discrete {H}eisenberg group
  and the norm of {H}arper's operator.
\newblock {\em J. Geom. Phys.}, 21(4):337--356, 1997.

\bibitem[Cla99]{MR1704205}
Bryan Clair.
\newblock Residual amenability and the approximation of {$L^2$}-invariants.
\newblock {\em Michigan Math. J.}, 46(2):331--346, 1999.

\bibitem[CP89]{MR1014114}
Massimo Campanino and J.~Fernando Perez.
\newblock Singularity of the density of states for one-dimensional chains with
  random couplings.
\newblock {\em Comm. Math. Phys.}, 124(4):543--552, 1989.

\bibitem[DS02]{Dicks_Schick}
Warren Dicks and Thomas Schick.
\newblock The spectral measure of certain elements of the complex group ring of
  a wreath product.
\newblock {\em Geom. Dedicata}, 93:121--137, 2002.

\bibitem[Dys53]{MR0059210}
Freeman~J. Dyson.
\newblock The dynamics of a disordered linear chain.
\newblock {\em Physical Rev. (2)}, 92:1331--1338, 1953.

\bibitem[Eck00]{Eckmann_intro}
Beno Eckmann.
\newblock Introduction to {$l_2$}-methods in topology: reduced
  {$l_2$}-homology, harmonic chains, {$l_2$}-{B}etti numbers.
\newblock {\em Israel J. Math.}, 117:183--219, 2000.
\newblock Notes prepared by Guido Mislin.

\bibitem[ER78]{PhysRevB.18.569}
T.~P. Eggarter and R.~Riedinger.
\newblock Singular behavior of tight-binding chains with off-diagonal disorder.
\newblock {\em Phys. Rev. B}, 18:569--575, Jul 1978.

\bibitem[ES05]{MR2178069}
G{\'a}bor Elek and Endre Szab{\'o}.
\newblock Hyperlinearity, essentially free actions and {$L^2$}-invariants.
  {T}he sofic property.
\newblock {\em Math. Ann.}, 332(2):421--441, 2005.

\bibitem[Fol95]{Folland:A_course_in_abstract_harmonic_analysis}
Gerald~B. Folland.
\newblock {\em A course in abstract harmonic analysis}.
\newblock Studies in Advanced Mathematics. CRC Press, Boca Raton, FL, 1995.

\bibitem[{Gra}10]{arXiv:grabowski-2010-2}
{\L}ukasz {Grabowski}.
\newblock {Irrational $l^2$-invariants arising from the lamplighter group},
  September 2010.
\newblock Preprint, available at \url{http://arxiv.org/abs/1009.0229}.

\bibitem[Gra14]{arxiv:grabowski-2010}
{\L}ukasz Grabowski.
\newblock On {T}uring dynamical systems and the {A}tiyah problem.
\newblock {\em Invent. Math.}, 198(1):27--69, 2014.

\bibitem[HJ90]{MR1084815}
Roger~A. Horn and Charles~R. Johnson.
\newblock {\em Matrix analysis}.
\newblock Cambridge University Press, Cambridge, 1990.
\newblock Corrected reprint of the 1985 original.

\bibitem[LL95]{lott-lueck-l2-topological-invariants-of-3-manifolds}
John Lott and Wolfgang L{\"u}ck.
\newblock L 2-topological invariants of 3-manifolds.
\newblock {\em Inventiones mathematicae}, 120(1):15--60, 1995.

\bibitem[LNW08]{Lehner_Neuhauser_Woess:On_the_spectrum_of}
Franz Lehner, Markus Neuhauser, and Wolfgang Woess.
\newblock On the spectrum of lamplighter groups and percolation clusters.
\newblock {\em Math. Ann.}, 342(1):69--89, 2008.

\bibitem[Lov12]{MR3012035}
L{\'a}szl{\'o} Lov{\'a}sz.
\newblock {\em Large networks and graph limits}, volume~60 of {\em American
  Mathematical Society Colloquium Publications}.
\newblock American Mathematical Society, Providence, RI, 2012.

\bibitem[LT14]{MR3110799}
Hanfeng Li and Andreas Thom.
\newblock Entropy, determinants, and {$L^2$}-torsion.
\newblock {\em J. Amer. Math. Soc.}, 27(1):239--292, 2014.

\bibitem[L{\"u}c94]{lueck_approximating_l2_invariants_by_their_finite_dimensional_analogues}
Wolfgang L{\"u}ck.
\newblock Approximating {$L^2$}-invariants by their finite-dimensional
  analogues.
\newblock {\em Geom. Funct. Anal.}, 4(4):455--481, 1994.

\bibitem[L{\"u}c02]{Lueck:Big_book}
Wolfgang L{\"u}ck.
\newblock {\em {$L^2$}-invariants: theory and applications to geometry and
  {$K$}-theory}, volume~44 of {\em Ergebnisse der Mathematik und ihrer
  Grenzgebiete. 3. Folge. A Series of Modern Surveys in Mathematics [Results in
  Mathematics and Related Areas. 3rd Series. A Series of Modern Surveys in
  Mathematics]}.
\newblock Springer-Verlag, Berlin, 2002.

\bibitem[L{\"u}c13]{2013arXiv1310.8564L}
Wolfgang L{\"u}ck.
\newblock {Estimates for spectral density functions of matrices over
  C[Z\^{}d]}, October 2013.
\newblock Preprint, available at \url{http://arxiv.org/abs/1310.8564}.

\bibitem[LW13]{arxiv:lehner_wagner-2010}
Franz Lehner and Stephan Wagner.
\newblock Free lamplighter groups and a question of {A}tiyah.
\newblock {\em Amer. J. Math.}, 135(3):835--849, 2013.

\bibitem[MM87]{MR914426}
F.~Martinelli and L.~Micheli.
\newblock On the large-coupling-constant behavior of the {L}iapunov exponent in
  a binary alloy.
\newblock {\em J. Statist. Phys.}, 48(1-2):1--18, 1987.

\bibitem[Pes08]{Pestov_Hyperlinear_and_sofic_groups_a_brief_guide}
Vladimir~G. Pestov.
\newblock Hyperlinear and sofic groups: a brief guide.
\newblock {\em Bull. Symbolic Logic}, 14(4):449--480, 2008.

\bibitem[PS{\,Z}10]{arxiv:pichot_schick_zuk-2010}
Mika{\"e}l Pichot, Thomas Schick, and Andrzej {\,Z}uk.
\newblock {Closed manifolds with transcendental L2-Betti numbers}.
\newblock {\em ArXiv e-prints}, May 2010.

\bibitem[RS80]{Reed_Simon_Methods_of_moder_mathematical_warfare_I}
Michael Reed and Barry Simon.
\newblock {\em Methods of modern mathematical physics. {I}}.
\newblock Academic Press Inc. [Harcourt Brace Jovanovich Publishers], New York,
  second edition, 1980.
\newblock Functional analysis.

\bibitem[Sau03]{Sauer(2003)}
Roman Sauer.
\newblock Power series over the group ring of a free group and applications to
  {N}ovikov-{S}hubin invariants.
\newblock In {\em High-dimensional manifold topology}, pages 449--468. World
  Sci. Publ., River Edge, NJ, 2003.

\bibitem[Sch01]{MR1828605}
Thomas Schick.
\newblock {$L^2$}-determinant class and approximation of {$L^2$}-{B}etti
  numbers.
\newblock {\em Trans. Amer. Math. Soc.}, 353(8):3247--3265, 2001.

\bibitem[Tho08]{thom_sofic_groups_and_diophantine_approximation}
Andreas Thom.
\newblock Sofic groups and {D}iophantine approximation.
\newblock {\em Comm. Pure Appl. Math.}, 61(8):1155--1171, 2008.

\end{thebibliography}
\end{document}